\documentclass{amsart}
\usepackage{color}
\usepackage{amsmath}
\usepackage{amssymb}
\usepackage{amsthm}
\usepackage{hyperref}
\usepackage{mathrsfs}
\usepackage{cite}
\usepackage[msc-links,abbrev]{amsrefs}
\usepackage[noabbrev,capitalize]{cleveref}
\usepackage{mathrsfs}
\usepackage{mathtools}
\usepackage[english]{babel}
\usepackage{enumerate}

\makeatletter

\def\p@enumii{\p@enumi\theenumi.}

\makeatother

\newcommand\restr[2]{{
  \left.\kern-\nulldelimiterspace 
  #1 
  \right|_{#2} 
  }}
\newcommand{\Ric}{\text{Ric}}

\newtheorem{theorem}{Theorem}[section]

\newtheorem{definition}{Definition}[theorem]

\title{Weighted GJMS operators on Smooth Metric Measure Spaces}
\author{Ayush Khaitan}
\address{316 Hill Center \\ Rutgers University \\ New Brunswick, NJ 08854 \\ USA}
\email{ayush.khaitan@rutgers.edu}

\begin{document}

\begin{abstract}
We construct weighted GJMS operators on smooth metric measure spaces, and prove that they are formally self-adjoint. We also provide factorization formulas for them in the case of quasi-Einstein spaces and under Gover--Leitner conditions. 
\end{abstract}

\maketitle
\section{Introduction}
Graham--Jenne--Mason--Sparling operators, commonly abbreviated as GJMS operators, are an important class of formally self-adjoint~\cites{GrahamZworski2003} conformally covariant operators~\cites{GJMS1992}. In this paper, we construct a weighted analogue of these operators for smooth metric measure spaces. 

A \emph{smooth metric measure space} is a five-tuple $(M^d,g,f,m,\mu)$, where $M^d$ is a Riemannian manifold of dimension $d$, $f$ a smooth function defined on $M$, $m\in\mathbb{R}$ a dimensional parameter, and $\mu\in\mathbb{R}_+$ an auxiliary curvature parameter~\cites{Case2014s,MR3415769}. If $m\in\mathbb{N}$, a smooth metric measure space may be thought of as the warped product $(M^d\times F^m(\mu),g\oplus f^2 h)$, where $(F^m(\mu),h)$ is the $m$-dimensional simply connected spaceform of constant curvature $\mu$~\cites{Case2014s}. 

The space of conformal densities of weight $w\in\mathbb{R}$ is denoted by $\mathcal{E}[w]$. Also, a pointwise conformal transformation of $(M^d,g,f,m,\mu)$ with respect to a smooth function $\sigma\in C^\infty(M)$ is the map $$(M^d,g,f,m,\mu)\mapsto (M^d,e^{2\sigma} g,e^\sigma f,m,\mu).$$  
{When $m>0$ we set $\phi:=-m\log f$ (and take $\phi=0$ if $m=0$) and define the weighted Laplacian $\Delta_\phi:=\Delta-\langle\nabla\phi,\nabla\cdot\rangle$.}

Weighted GJMS operators are known in orders two and four~\cites{Case2011t}, and were formally defined by Case and Chang~\cites{CaseChang2013} to study fractional GJMS operators~\cites{GrahamZworski2003} via a curved analogue of the Caffarelli--Silvestre extension~\cites{CaffarelliSilvestre2007,ChangGonzalez2011, ChangYang2017, Case2019tr}. In this paper, we give a rigorous definition of the weighted GJMS operators, and develop some of their properties.

Weighted GJMS operators are the canonical conformally invariant operators defined on smooth metric measure spaces, with a power of the laplacian as the leading term. On setting $m=0$, we recover the GJMS operators on Riemannian manifolds~\cites{GJMS1992}. On taking the limit $m\to\infty$, we recover~\cite{KHAITAN2026110787} Perelman's modified Bochner-Lichnerowicz formula~\cite{Perelman1}. 

The ambient metric is a key tool in defining weighted GJMS operators, and a weighted analogue of the ambient metric has recently been defined by Case and the author~\cites{CaseKhaitan2022}. By adapting the arguments in \cites{GJMS1992}, we construct the weighted GJMS operators. 

\begin{theorem}
\label{main-theorem}
If $d+m\notin 2\mathbb{N}$, then for each positive integer $k$ there is a conformally invariant operator $$L_{2k,\phi}^m:\mathcal{E}[-\frac{1}{2}(d+m)+k]\to \mathcal{E}[-\frac{1}{2}(d+m)-k],$$ with leading term $(\Delta_\phi)^k$. If $d+m\in 2\mathbb{N}$, the same result holds with the restriction $1\leq k\leq \frac{1}{2}(d+m)$.
\end{theorem}

Graham--Zworski proved the formal self-adjointness of GJMS operators~\cites{GrahamZworski2003}. Other proofs for formal self-adjointness are also known~\cites{FeffermanGraham2002,FeffermanGraham2013,Juhl2013}. In this article, we prove that weighted GJMS operators are formally self-adjoint, closely following the method used in~\cites{FeffermanGraham2002}.
\begin{theorem}
\label{self-adjoint-theorem}
$L_{2k,\phi}^m$ is a formally self-adjoint operator. 
\end{theorem}
Our proof of formal self-adjointness follows the method of~\cites{FeffermanGraham2002}.

GJMS operators factor nicely for Einstein metrics~\cites{Gover2006q,FeffermanGraham2012}. Moreover, Case and Chang~\cites{CaseChang2013} have also formally proved a factorization of these operators for products of negatively-curved Einstein manifolds with positively-curved Einstein manifolds, using the explicit ambient metric for such spaces found by Gover and Leitner~\cites{GoverLeitner2008}. In this paper, we provide factorization formulas of GJMS operators for quasi-Einstein spaces~\cites{Case2014s} and under Gover--Leitner conditions~\cites{GoverLeitner2008}. Before we state our result, we provide the relevant definitions and formulas. 

Let $(M^d,g,f,m,\mu)$ be a smooth metric measure space. When $m>0$, we set $\phi:=-m\log f$, so that
\[ dv_\phi = e^{-\phi}\,\operatorname{dvol}_g . \] {We define the $m$-weighted Bakry--Emery Ricci tensor and the associated weighted scalar curvature by}
\begin{align*}
\Ric_\phi^m &= \Ric + \nabla^2\phi - \frac{1}{m} d\phi \otimes d\phi,\\
R_\phi^m &= R + 2\Delta\phi - \frac{m+1}{m}|\nabla\phi|^2 + m(m-1)\mu e^{2\phi/m} .
\end{align*}
{When $m\in\mathbb{N}$, these coincide with the Ricci tensor and scalar curvature of the warped product $(M^d\times F^m(\mu),g\oplus f^2 h)$ restricted to $M$, so they are the natural curvature quantities in this setting. For brevity, we suppress the dependence of $R_\phi^m$ on $\mu$ in the notation.}
{We also use the weighted scalar function \[F_\phi^m:=f\Delta f+(m-1)(|\nabla f|^2-\mu).\]}
Recall that \emph{weighted Schouten tensor $P_\phi^m$} and the \emph{weighted Schouten scalar $J_\phi^m$} of $(M^d,g,f,m,\mu)$ are
\begin{align*}
    P_\phi^m & := \frac{1}{d+m-2}(\Ric_\phi^m - J_\phi^m g), \\
    J_\phi^m & := \frac{1}{2(d+m-1)}R_\phi^m .
\end{align*}
A \emph{quasi-Einstein space}~\cites{Case2014s} is a smooth metric measure space such that for some $\lambda\in\mathbb{R}$, \[P_\phi^m=\lambda g, \quad J_\phi^m=(d+m)\lambda.\] In this paper, we prove a factorization formula of the weighted GJMS operator for quasi-Einstein spaces.
\begin{theorem}
\label{quasi-einstein-theorem}
The weighted GJMS operator
\begin{equation*}L_{2k,\phi}^m: \mathcal{E}\left[-\frac{d+m}{2}+k\right] \rightarrow \mathcal{E}\left[-\frac{d+m}{2}-k\right]\end{equation*}
can be factorized as
\begin{equation*}
\prod_{l=0}^{k-1}\left[\Delta_\phi+2\lambda\left(-\frac{d+m}{2}+k-2l\right)\left(\frac{d+m}{2}+k-1-2l\right)\right]
\end{equation*}
for quasi-Einstein spaces.
\end{theorem}
\emph{Weighted Gover--Leitner conditions} are a generalization to $m\notin \mathbb{N}_0$ of the Gover--Leitner conditions defined in~\cites{GoverLeitner2009}, and are defined as
\begin{equation*}
f(x)=1\quad
\mu=1,\quad
(\text{Ric}_\phi^m)_{ij}=(-d+1)g.
\end{equation*} {Since $f\equiv 1$ here, we have $\phi=0$, hence 
\begin{align*}
\Ric_\phi^m=\Ric,\quad \Delta_\phi=\Delta,\quad R_\phi^m&=R+m(m-1).
\end{align*}}
In this paper, we also prove a factorization formula of the weighted GJMS operator for smooth metric measure spaces under weighted Gover--Leitner conditions. We thus provide a rigorous proof of the factorization formula constructed for Poincar\'e-Einstein spaces by Case and Chang~\cites{CaseChang2013}. {That is, we obtain a canonical construction and factorization of the operators by passing to an ambient space, instead of introducing them ad hoc.}
\begin{theorem}
\label{gover-leitner-theorem}
The weighted GJMS operator
\begin{equation*}L_{2k,\phi}^m: \mathcal{E}\left[-\frac{d+m}{2}+k\right] \rightarrow \mathcal{E}\left[-\frac{d+m}{2}-k\right]\end{equation*}
can be factorized as
\begin{equation*}
L_{2k,\phi}^m =\prod_{j=0}^{k-1}\left[\Delta+\frac{(2 k-4 j-d-m)(2-d+m-2 k+4 l)}{4}\right]
\end{equation*}
under Gover--Leitner conditions.
\end{theorem}

This article is organized as follows: in \cref{sec:twas}, we discuss some properties of the weighted ambient space and the commutation relations satisfied by the differential operators relevant to this article. In \cref{sec:wcio}, we derive weighted GJMS operators with the help of the weighted ambient space. In \cref{sec:cwgo}, we show that weighted GJMS operators have a power of the weighted Laplacian as the leading part. In \cref{sec:ff}, we derive the factorization formulas for weighted GJMS operators in the case of quasi-Einstein spaces, and also under Gover--Leitner conditions. In \cref{sec:sa}, we prove that weighted GJMS operators are formally self-adjoint. 

\subsection*{Acknowledgements}
I would like to thank Professor Jeffrey S. Case for suggesting this problem, many helpful discussions, and for comments on multiple drafts of the paper. I would also like to thank the anonymous referees for making several helpful suggestions that have substantially improved the paper.

\section{The weighted ambient space}
\label{sec:twas}
For a smooth metric measure space $(M^d,g,f,m,\mu)$, consider the $(d+2)$-dimensional space $\mathbb{R}_+\times M\times \mathbb{R}$ with coordinates $(t,x,\rho)$. Then the corresponding \emph{straight and normal weighted ambient space} is
\begin{equation}
\begin{aligned}
\label{normalmetric}
\widetilde{g}& :=t^2 g_\rho+2\rho dt^2+2tdtd\rho,\\
\widetilde{f}& :=tf_\rho,
\end{aligned}
\end{equation} such that $\widetilde{\text{Ric}}_\phi^m,\widetilde{F_\phi^m}=O(\rho^j)$ and $\mathrm{tr}\,\widetilde{\text{Ric}}_\phi^m-mf^{-2}\widetilde{F_\phi^m}=O(\rho^{j+1})$, where $j=(d+m-2)/2$ or $\infty$ depending on whether $d+m\in 2\mathbb{N}$ or $d+m\notin 2\mathbb{N}$. {Here $\widetilde{F_\phi^m}:=\widetilde{f}\widetilde{\Delta}\widetilde{f}+(m-1)(|\widetilde{\nabla}\widetilde{f}|^2-\mu)$.} The existence, and the uniqueness of $(\widetilde{g},\widetilde{f})$ up to order $O(\rho^j)$ and of $[(1/2)g^{kl}(g_\rho)_{kl}+(m/f)(f_\rho)]$ up to order $O(\rho^{j+1})$, has been proven in \cites{CaseKhaitan2022}. {Here $[(1/2)g^{kl}(g_\rho)_{kl}+(m/f)(f_\rho)]$ may be thought of as the weighted version of the trace of $g_\rho$. This is analogous to how the trace of $g_\rho$ is uniquely determined to one higher order in \cites{FeffermanGraham2012}.} We denote $\mathbb{R}_+\times M\times \mathbb{R}$ as $\widetilde{\mathcal{G}}$, the position vector on $\widetilde{\mathcal{G}}$ as $X^I$, and $\restr{\widetilde{\mathcal{G}}}{\rho=0}$ as $\mathcal{G}$.

\subsection{Commutation relations}
\cref{normalmetric} implies that
\begin{equation}
\label{gjms4}
\widetilde{\nabla}_IX_J=\widetilde{g}_{IJ}.
\end{equation} Note that \cref{gjms4} can also be deduced from the fact that \cref{normalmetric} is a straight weighted ambient metric; the definition of a straight weighted ambient metric is given in~\cites{CaseKhaitan2022}.
From \cref{gjms4}, we get that $\widetilde{\nabla}_K\widetilde{\nabla}_J X_I-\widetilde{\nabla}_J\widetilde{\nabla}_K X_I=0$. Hence,  
\begin{equation}
\label{gjms5}
\widetilde{R}_{LKJI}X^L=0.
\end{equation}
Now set $Q=X^IX_I$. From \cref{normalmetric}, we conclude that it is a defining function for $\mathcal{G}$. From \cref{gjms4}, we compute that 
\begin{equation}
\label{gjms6}
\widetilde{\nabla}_I Q=2X_I.
\end{equation}

Let us now define the following operators acting on functions on $\widetilde{\mathcal{G}}$.
\begin{align*}
\begin{split}
x=-\frac{1}{4}Q,\quad  y=\widetilde{\Delta}_\phi,\quad 
h=X+\frac{1}{2}(d+m+2).
\end{split}
\end{align*}
{Here $\widetilde{\Delta}_\phi$ is the weighted Laplacian, $\widetilde{\Delta}_\phi:=\widetilde{\Delta}-\widetilde{\nabla}_{\widetilde{\nabla}\phi}$.} Note that the degree of homogeneity of $f$ with respect to $t$ is $1$, and $X=t\partial_t$. 
\begin{theorem}
The operators $x,y,h$ satisfy the commutation relations
\begin{equation*}
[x,y]=h,\quad 
[h,x]=2x,\quad 
[h,y]=-2y.
\end{equation*}
\end{theorem}
\begin{proof}
First, we prove that $[x,y]=h$. Let $F$ be a function defined on $\widetilde{\mathcal{G}}$. Using \cref{gjms4}, \cref{gjms6} and the fact that $\widetilde{g}^{IJ}\widetilde{g}_{IJ}=d+2$, we get $$[x,y]F=[X+\frac{1}{2}(d+2)](F)-\frac{1}{4}(\widetilde{\nabla}^I\phi\widetilde{\nabla}_IQ)F.$$ As $\widetilde{\nabla}^I\phi\widetilde{\nabla}_I Q=-2\frac{m}{\widetilde{f}}\widetilde{\nabla}_X \widetilde{f}=-2m$, we have $-\frac{1}{4}(\widetilde{\nabla}^I\phi\widetilde{\nabla}_IQ)F=\frac{1}{2}mF$. Consequently, $[x,y]=h$.

Second, we prove that $[h,x]=2x$. This follows from the fact that $XQ=2Q$, as $Q$ is homogeneous of degree $2$.

Third, we prove that $[h,y]=-2y$. Using \cref{gjms5}, we conclude that 
\begin{align*}
[X,\widetilde{\Delta}](F)&=-2\widetilde{\Delta}(F),\\
[\widetilde{\nabla}_{\widetilde{\nabla}\phi},X](F)&=\widetilde{\nabla}_{\widetilde{\nabla}\phi}(F)-X^I(\widetilde{\nabla}^J\widetilde{\nabla}_I\phi)(\widetilde{\nabla}_J F).
\end{align*}On commuting $X_I$ with $\widetilde{\nabla}^J$, using \cref{gjms4}, and noting that $\widetilde{\nabla}_X\phi=-m$, we get {$[\widetilde{\nabla}_{\widetilde{\nabla}\phi},X](F)=2\widetilde{\nabla}_{\widetilde{\nabla}\phi}(F)$.} Hence, $[h,y]=-2y$.\end{proof}

Using an induction argument, we get the following commutation relations (cf.\ \cites{GJMS1992}).
\begin{subequations}
\begin{align}
\begin{split}
\label{liebracket1}
[y^k,x]&=-ky^{k-1}(h-k+1),
\end{split}\\
\begin{split}
\label{liebracket2}
[x^k,y]&=kx^{k-1}(h+k-1),
\end{split}\\
\begin{split}
\label{liebracket3}
y^{k-1}x^{k-1}&=(-1)^{k-1}(k-1)!h(h+1)\dots (h+k-2)+xZ_k,
\end{split}
\end{align}
\end{subequations} for some polynomial $Z_k$ in $x,y,h$.

\section{Weighted conformally invariant operators}
\label{sec:wcio}
In this section, we construct two weighted GJMS operators, and then prove that they are the same up to a constant.

In the rest of the paper, $w=-(d+m)/2+k$. 
\begin{theorem}
Let $k\in\mathbb{N}$ and $F\in \mathcal{E}[w]$. Then $\restr{\widetilde{\Delta}_\phi^k\widetilde{F}}{{\mathcal{G}}}$ is independent of the choice of $\widetilde{F}$, where $\widetilde{F}$ is a smooth homogeneous extension of $F$ to $\widetilde{\mathcal{G}}$. Thus, $L_{2k,\phi}^m:\mathcal{E}[w]\mapsto \mathcal{E}[w-2k]$, defined as $L_{2k,\phi}^m F:=\restr{\widetilde{\Delta}_\phi^k\widetilde{F}}{\mathcal{G}}$, is a conformally invariant operator.
\end{theorem}
\begin{proof}
Any two extensions of $F$ differ by a function of the form $QH$, where $H\in C^\infty (\widetilde{\mathcal{G}})$ is homogeneous of weight $w-2$. Now by \cref{liebracket1}, we have \[\widetilde{\Delta}_\phi^k(QH)=Q\widetilde{\Delta}_\phi^k(H)+4k\widetilde{\Delta}_\phi^{k-1}(w+\frac{1}{2}(d+m)-k)H=Q\widetilde{\Delta}_\phi^k(H).\] As $\widetilde{\Delta}_\phi$ reduces the degree of homogeneity by $2$, it holds that $\restr{\widetilde{\Delta}_\phi^k(F)}{{\mathcal{G}}}$ belongs to $\mathcal{E}[w-2k]$.

\end{proof}
We now study the obstruction to constructing harmonic extensions of smooth conformal densities on $\mathcal{G}$.
\begin{theorem}
\label{gjmstheorem2}
For $F\in\mathcal{E}[w]$,
\begin{enumerate}
\item if $k\notin\mathbb{N}$, then $F$ has a unique formal harmonic extension to $\widetilde{\mathcal{G}}$, homogeneous of degree $w$;
\item if $k\in\mathbb{N}$, then $F$ has a homogeneous extension $\widetilde{F}$ uniquely determined modulo $O(Q^k)$ by the requirement that $\widetilde{\Delta}_\phi\widetilde{F}=0$ modulo $O(Q^{k-1})$. The obstruction to a harmonic extension is $\restr{Q^{1-k}\widetilde{\Delta}_\phi^k\widetilde{F}}{{\mathcal{G}}}$, which is independent of the extension $\widetilde{F}$ mod $Q^k$, and hence is conformally invariant.
\end{enumerate}
\end{theorem}
\begin{proof}
Assume that we have an extension $\widetilde{F}_{l-1}$ of $F$ such that $\widetilde{\Delta}_\phi\widetilde{F}_{l-1}=0\bmod Q^{l-1}$. Now let $\widetilde{F}_l=\widetilde{F}_{l-1}+Q^l H$, where $H$ is of weight $w-2l$. We have
\begin{equation*}
\begin{aligned}
\widetilde{\Delta}_\phi\widetilde{F}_l&=\widetilde{\Delta}_\phi\widetilde{F}_{l-1}+\widetilde{\Delta}_\phi(Q^l H)\\
&=\widetilde{\Delta}_\phi\widetilde{F}_{l-1}+4lQ^{l-1}(k-l)H\bmod Q^l,
\end{aligned}
\end{equation*}
If $k$ is not a positive integer, we can choose a unique function $H$ for each $l$ such that the above expression is $0\bmod Q^l$. On the other hand, if $k$ is a positive integer, then $\widetilde{\Delta}_\phi(\widetilde{F}_k)=\widetilde{\Delta}_\phi(\widetilde{F}_{k-1}) \bmod Q^k$. Note that $\restr{Q^{1-k}\widetilde{\Delta}_\phi(\widetilde{F})}{{\mathcal{G}}}$ depends only on $F$, and is homogeneous of degree $w-2k$. 
\end{proof}
We now show that the two conformally invariant operators constructed above are scalar multiples of one another.
\begin{theorem}
If $k\in\mathbb{N}$, then
\begin{equation}
\label{multiplicative-constant}
\restr{\widetilde{\Delta}_\phi^k\widetilde{F}}{{\mathcal{G}}}=(-4)^{k-1}(k-1)!^2 \restr{Q^{1-k}\widetilde{\Delta}_\phi \widetilde{F}}{{\mathcal{G}}},
\end{equation} where $\widetilde{F}$ is the extension of $F$ such that $\widetilde{\Delta}_\phi\widetilde{F}=0\bmod Q^{k-1}$.  
{Note that the right-hand side involves a single weighted Laplacian; the power $k$ appears only on the left-hand side, as dictated by the commutation relations.}
\end{theorem}

\begin{proof}
Let $L=\restr{Q^{1-k}\widetilde{\Delta}_\phi}{{\mathcal{G}}}$ be as in \cref{gjmstheorem2}. Then $\widetilde{\Delta}_\phi\widetilde{F}_{k-1}=Q^{k-1}LF \bmod Q^k$. Now from \cref{liebracket3}, we know that 
\begin{equation*}
   \widetilde{\Delta}_\phi^k\widetilde{F}_{k-1}=\widetilde{\Delta}_\phi^{k-1}(Q^{k-1}LF)=4^{k-1}(k-1)!h(h+1)\dots(h+k-2) LF\bmod Q.
\end{equation*} 
But $hLF=-(k-1)LF$. Using this, we verify \cref{multiplicative-constant}.
\end{proof}
In the rest of the paper, we shall denote $\restr{\widetilde{\Delta}_\phi^k\widetilde{F}}{{\mathcal{G}}}$ as $L^m_{2k,\phi}F$.

\section{The leading order term}
\label{sec:cwgo}
Let $k\in\mathbb{N}$. For $\psi\in C^\infty(M)$, let $t^w\psi(x,\rho)$ be a homogeneous extension of weight $w$. The weighted Laplacian with respect to the ambient metric measure structure of the form \cref{normalmetric} is
\begin{multline}
\label{tildelaplacian}
\widetilde{\Delta}_\phi(t^w\psi)=t^{w-2}\big[-2\rho\psi''+(2w+d+m-2-\rho g^{ij}g'_{ij})\psi'\\+
\Delta_\phi\psi+\frac{1}{2}w\psi g^{ij}g'_{ij}
+\frac{m}{f}f'(w\psi-2\rho\psi')\big],
\end{multline}
where $\psi=\psi(x,\rho)$, the prime denotes $\partial_{\rho}$ and the $\Delta_\phi$ on the right-hand side refers to the weighted Laplacian with respect to $(g_\rho,f_\rho)$. 

{Now let us assume that $t^w\psi$ is a harmonic extension of the form constructed in \cref{gjmstheorem2}, i.e. $\widetilde{\Delta}_\phi(t^w\psi)=0$ modulo $O(Q^{k-1})$. On differentiating the identity $\widetilde{\Delta}_\phi(t^w\psi)=0$ with respect to $\rho$ a total of $l$ times, where $l<k-1$, and setting $\rho=0$, we get}
\begin{multline}
\label{expanded-laplacian-derivative}
2(l+1-k)\restr{(\partial_{\rho})^{l+1}}{\rho=0}\psi=\restr{(\partial_{\rho})^l}{\rho=0}\big[\Delta_\phi\psi-2\rho\psi'( \frac{1}{2}g^{ij}g'_{ij}+\frac{m}{f}f')\\
+w\psi(\frac{1}{2}g^{ij}g'_{ij}+\frac{m}{f}f')\big].
\end{multline} {Here the coefficient $2(l+1-k)$ comes from combining the contribution of the term $-2\rho\psi''$ (which yields $-2l\,\partial_\rho^{l+1}\psi|_{\rho=0}$) with the $(2w+d+m-2)\psi'$ term, using $2w+d+m-2=2(k-1)$.} For $l=k-1$, we get 
\begin{equation}
\label{k-1}    
c_k L_{2k,\phi}^m \psi=\restr{(\partial_{\rho})^{k-1}}{\rho=0}\big[\Delta_\phi\psi-2\rho\psi'( \frac{1}{2}g^{ij}g'_{ij}+\frac{m}{f}f')+w\psi(\frac{1}{2}g^{ij}g'_{ij}+\frac{m}{f}f')\big],
\end{equation} {This follows from \cref{multiplicative-constant}: the ambient obstruction $\restr{Q^{1-k}\widetilde{\Delta}_\phi^k(t^w\psi)}{\mathcal{G}}$ is a constant multiple of $L_{2k,\phi}^m\psi$, and evaluating the obstruction at $\rho=0$ corresponds to the $(k-1)$st $\rho$-derivative of the bracketed term.} where $c_k=(-1)^{k-1}[2^{k-1}(k-1)!]^{-1}$. 

Let $d+m\in 2\mathbb{N}$, in which case $(g_\rho,f_\rho)$ is uniquely defined modulo $O(\rho^{\frac{d+m}{2}})$, while $( \frac{1}{2}g^{ij}(g_\rho)_{ij}+\frac{m}{f}f_\rho)$ is uniquely defined modulo $O(\rho^{\frac{d+m}{2}+1})$~\cite{CaseKhaitan2022}. In \cref{expanded-laplacian-derivative}, as $l\leq k-2$, we have {up to} $k-1$ derivatives of $( \frac{1}{2}g^{ij}(g_\rho)_{ij}+\frac{m}{f}f_\rho)$ on the right. Hence, if $k-1\leq (d+m)/2$, the right side of \cref{expanded-laplacian-derivative} does not depend on the ambiguity of $(\widetilde{g},\widetilde{f})$, and can be uniquely expressed in terms of the derivatives of $\psi,\widetilde{g}$ and $\widetilde{f}$. However, in \cref{k-1}, there are $k-1$ and lower derivatives of $\psi$, and $k$ and lower derivatives of $( \frac{1}{2}g^{ij}(g_\rho)_{ij}+\frac{m}{f}f_\rho)$ at $\rho=0$. Hence, both \cref{expanded-laplacian-derivative,k-1} are independent of the ambiguity of $(\widetilde{g},\widetilde{f})$ only for $k\leq (d+m)/2$. Also, note that $w=0$ for $k=(d+m)/2$.  

\cref{k-1} also tells us that $L_{2k,\phi}^m$ has leading part $(\Delta_\phi)^k$. {Indeed, iterating \cref{expanded-laplacian-derivative} expresses each $\partial_\rho^{j}\psi|_{\rho=0}$ in terms of $\Delta_\phi^{j}\psi$ plus lower-order terms, so the top-order contribution comes from $(\Delta_\phi)^k$.} This completes the proof of \cref{main-theorem}.








\section{Factorization formulas}
\label{sec:ff}
We now prove factorization formulas of the weighted GJMS operator $\restr{\widetilde{\Delta}_\phi^k}{{\mathcal{G}}}$ under quasi-Einstein conditions and Gover--Leitner conditions.

\subsection{Quasi-Einstein conditions}
\begin{proof}[Proof of \cref{quasi-einstein-theorem}]
Let \begin{equation}\label{quasi-einstein-form}g_\rho(x)=(1+\lambda\rho)^2 g(x),\quad f_\rho(x)=(1+\lambda\rho)f(x).\end{equation} We know from (\cites{CaseKhaitan2022}, Section 7) that $(\widetilde{g},\widetilde{f})$ of the form \cref{normalmetric}, with $(g_\rho,f_\rho)$ as given in \cref{quasi-einstein-form}, is a weighted ambient space.  Also, for $\psi\in C^\infty(M)$, let $\widetilde{\psi}(t,x,\rho)=t^w (1+\lambda\rho)^w \psi(x).$ \cref{tildelaplacian} becomes
\begin{equation*}
    \widetilde{\Delta}_\phi \widetilde{\psi}=t^{w-2}(1+\lambda\rho)^{w-2}\left[\Delta_\phi+2w\lambda(w+d+m-1)\right]\psi.
\end{equation*}
By induction, we obtain
\begin{equation*}
\widetilde{\Delta}^k_\phi\widetilde{\psi}= t^{w-2k}(1+\lambda\rho)^{w-2k}\prod\limits_{l=0}^{k-1}\left[\Delta_\phi+2\lambda(w-2l)(w+d+m-1-2l)\right]\psi(x).
\end{equation*}
With $w=-(d+m)/2+k$, restricting to $\rho=0$, and using the fact that $\restr{\widetilde{\Delta}^k_\phi\widetilde{\psi}}{\mathcal{G}}$ is independent of the choice of extension $\widetilde{\psi}$ to $\widetilde{\mathcal{G}}$, we get 
\begin{equation*}
L^m_{2k,\phi}(\psi)=\prod_{l=0}^{k-1}\left[\Delta_\phi+2\lambda\left(-\frac{d+m}{2}+k-2l\right)\left(\frac{d+m}{2}+k-1-2l\right)\right]\psi.\qedhere
\end{equation*} 
\end{proof}
The idea for such an argument originated in \cites{Matsumoto2013}.

\subsection{ Gover--Leitner conditions}
\begin{proof}[Proof of \cref{gover-leitner-theorem}]
Let $(g_\rho,f_\rho)$ be defined {by $g_\rho(x)=g(x,\rho)$ and $f_\rho(x)=f(x,\rho)$, where}
\begin{equation*}
g(x,\rho)=\big(1-\frac{1}{2}\rho\big)^2 g(x),\quad  f(x,\rho)=\big(1+\frac{1}{2}\rho\big). 
\end{equation*}
We know from (\cites{CaseKhaitan2022}, Section 7) that $(\widetilde{g},\widetilde{f})$ of the form \cref{normalmetric}, with $(g_\rho,f_\rho)$ {as defined above}, is a weighted ambient space.
Now if $\widetilde{\psi}(x,\rho,t)=t^w\left(1-\rho/2\right)^{w} \psi(x)$ for 
$w=-(d+m)/2+k$, on using \cref{tildelaplacian} and induction we get 
\begin{equation*}
\restr{\widetilde{\Delta}^k_\phi\widetilde{\psi}}{\mathcal{G}} =t^{-\frac{d+m}{2}-k}\prod_{j=0}^{k-1}\left[\Delta+\frac{(2 k-4 j-d-m)(2-d+m-2 k+4 l)}{4}\right]\psi.
\end{equation*}
Since $\restr{\widetilde{\Delta}^k_\phi\widetilde{\psi}}{\mathcal{G}}$ is independent of the choice of extension $\widetilde{\psi}$ to $\widetilde{\mathcal{G}}$, we get 
\begin{equation*}
L_{2k,\phi}^m =\prod_{j=0}^{k-1}\left[\Delta+\frac{(2 k-4 j-d-m)(2-d+m-2 k+4 l)}{4}\right].\qedhere
\end{equation*}
\end{proof}

\section{Formal self-adjointness}
\label{sec:sa}
First, we recall the definition of a weighted Poincar\'e space. 
\begin{definition}
\label{poincare-definition}
A \emph{weighted Poincar\'e space} for $(M^d,[g,f],m,\mu)$, $m<\infty$, is a metric measure structure $(g_+,f_+)$ on $M\times [0,\epsilon)$ such that
\begin{enumerate}[(i)]
\item $g_+$ has signature $(d+1,0)$;
\item $(g_+,f_+)$ has $(M^d,[g,f],m,\mu)$ as conformal infinity; and

\item \begin{enumerate}[(a)]
\item if $d+m\notin 2\mathbb{N}$, then 
\begin{equation}
\begin{aligned}
\big(\text{Ric}_{\phi}^m(g_+)+(d+m)g_+,
F_\phi^m(g_+)
-(d+m)f_+^2\big)=O(r^\infty);\end{aligned} 
\end{equation}

\item if $d+m\in 2\mathbb{N}$, then 
\begin{equation}
\begin{aligned}
\big(\text{Ric}_{\phi}^m(g_+)+(d+m)g_+,
F_\phi^m(g_+)
-(d+m)f_+^2\big)=O^{2,+}_{\alpha\beta}(r^{d+m-2}).\end{aligned} 
\end{equation}
\end{enumerate}
\end{enumerate}
\end{definition}
{Here $(A,B)=O(r^j)$ denotes that both $A$ and $B$ vanish to order $r^j$; $O^{2,+}_{\alpha\beta}(r^{d+m-2})$ denotes $2$-tensors that vanish to order $r^{d+m-2}$, but their trace vanishes to one higher order.}
A weighted Poincar\'e space $(M^d,[g,f],m,\mu)$, $m<\infty$, is in \emph{normal form relative to $(g,f)$} if $g_+=r^{-2}(dr^2+g_r)$ and $f_+=r^{-1}f_r$. Here $g_r$ is a one-parameter family of metrics on $M$ such that $g_0=g$, and $f_r$ is a one-parameter family of functions such that $f_0=f$.

We now prove \cref{self-adjoint-theorem}. To begin, we identify the weighted GJMS operators in terms of the weighted Poincar\'e space~\cites{CaseKhaitan2022}.
\begin{theorem}
\label{differential-equation-solution}
Let $(X^{d+1}, g_+,f_+,m,\mu)$ be a weighted Poincar\'e space \cites{CaseKhaitan2022} for the smooth metric measure space $(M^d,g,f,m,\mu)$, and let $v \in C^{\infty}(M) .$ Also, let $k \in \mathbb{N}$ and $s=(d+m)/2+k$, with $k\leq (d+m)/2$ if $d+m\in 2\mathbb{N}$. Then there is a formal solution to the equation 
\begin{equation*}
\label{differential-equation}
\left((\Delta_{\phi_{+}}){}_{g_+}-s(d+m-s)\right) u=O\left(r^{2k}\log r\right)
\end{equation*}
of the form
\begin{equation*}
\label{4.2}
u=r^{\frac{d+m}{2}-k}\left(V+ d_k L_{2k,\phi}^m v\, r^{2k} \log r \right),
\end{equation*}
where the function $V \in C^{\infty}(\overline{X})$ is uniquely determined by $g,f$ and $v$ modulo $O\left(r^{2k}\right)$, $\restr{V}{{M}}=v$, and $d_k=[2^{2k-1}k!(k-1)!]^{-1}$.

\end{theorem}
\begin{proof}
For a chosen $(g,f)$, we may assume without loss of generality~{\cites{CaseKhaitan2022} (Proposition 5.3)} that we have a normal weighted Poincar\'e metric, i.e. $(g_+,f_+)=(r^{-2}\left(g_{r}+d r^{2}\right),r^{-1}f_r)$. A straightforward calculation shows that $\left[(\Delta_{\phi_+})_{g_+}-s(d+m-s)\right] \circ r^{d+m-s}=r^{d+m-s+1} \mathcal{D}_{s}$, where
\begin{multline}
\label{D_s}
\mathcal{D}_{s}=-r \partial_{r}^{2}+\left[2 s-d-m-1-r\left(\frac{1}{2} g^{i j} g_{i j}^{\prime}+\frac{m}{f}f'\right)\right] \partial_{r}\\
-(d+m-s)\left(\frac{1}{2} g^{i j} g_{i j}^{\prime}+\frac{m}{f}f'\right)+r (\Delta_\phi)_{g_{r}}.
\end{multline}
Here $(g_{i j},f)$ denotes $(g_{r},f_r)$ with $r$ fixed, and $(g_{i j}^{\prime},f^{\prime})=(\partial_{r} g_{i j},\partial_r f)$. For $v_{j} \in$ $C^{\infty}(M)$ and $s=(d+m)/2+k$, one has
\begin{equation*}
\mathcal{D}_{s}\left(v_{j} r^{j}\right)=j(2k-j) v_{j} r^{j-1}+O\left(r^{j}\right).
\end{equation*}
Beginning with $V_{0}=v$, define $v_{j}, V_{j}$ for $j \geq 1$ by
\begin{equation*}
\label{4.6}
\begin{aligned}
j(2k-j) v_{j}&=-\left.\left(r^{1-j} \mathcal{D}_{s}\left(V_{j-1}\right)\right)\right|_{x=0}, \\
V_{j}&=V_{j-1}+v_{j} r^{j}.
\end{aligned}
\end{equation*} Observe that since $(g_{r},f_r)$ is even in $r$ modulo $O(r^j)$, where $j=\infty$ or $d+m$ for $d+m\notin 2\mathbb{N}$ or $d+m\in 2\mathbb{N}$ respectively~\cites{CaseKhaitan2022}, $\mathcal{D}_{s}$ maps even functions to odd and vice versa modulo $O(r^j)$. Therefore, $v_{j}=0$ for $j$ odd and $j<d+m$. 

For $j=2k$, there is an obstruction to solving for a smooth function $V$. However, observe that
\begin{align*}
\mathcal{D}_{s}\left(p_{j} r^{j} \log r\right)=j (2 k-j) p_{j} r^{j-1} \log r+2(k- j) p_{j} r^{j-1} 
+O\left(r^{j} \log r\right).
\end{align*}
Therefore, if we take
$$
p_{2 k}=\left.(2 k)^{-1}\left(r^{1-2 k} \mathcal{D}_{s}\left(V_{2 k-1}\right)\right)\right|_{r=0},
$$
then we have $\mathcal{D}_{s} V_{2 k}=O\left(r^{2 k} \log r\right)$, and $V_{2k}=V_{2k-1}+p_{2k}r^{2k}\log r$. We can deduce from \cref{D_s} that $p_{2k}=d_kP_{2k,\phi}^m v$, where $P_{2k,\phi}^m$ is a differential operator with principal part $\Delta_\phi^k$, and $d_k=[2^{2k-1}k!(k-1)!]^{-1}$.

Now we show that $P_{2k,\phi}^m$, when defined in terms of the ambient metric, is the same as the differential operator $L_{2k,\phi}^m$.

Let $x=\sqrt{-2\rho}$ and $v=xt$. Then from {\cites{CaseKhaitan2022} (Proposition 5.6)} we know that 
\begin{equation*}
\begin{aligned}
\widetilde{g}&=-dv^2+v^2 g_+,\\ 
\widetilde{f}&=vf_+,    
\end{aligned}
\end{equation*} where $(\widetilde{g},\widetilde{f})$ is a straight and normal weighted ambient space and $(g_+,f_+)$ is a normal weighted Poincar\'e space. For function $\widetilde{F}$ of weight $w$, we find through direct computation that 
\begin{equation*}
\widetilde{\Delta}_{\phi}\widetilde{F}=v^{-2}\big[(\Delta_{\phi_+}){}_{g_{+}}+w(w+d+m)\big]\widetilde{F}.    
\end{equation*}
Let $s=w+d+m$. Then this equation can be written as 
\begin{equation*}
\widetilde{\Delta}_{\phi}\widetilde{F}=v^{-2}\big[(\Delta_{\phi_+}){}_{g_+}-s(d+m-s)\big]\widetilde{F}.    
\end{equation*}
Let $u$ be the restriction of $\widetilde{F}$ to the Poincar\'e-Einstein space $v=1$. Then $\widetilde{F}$ can be recovered from $u$ by $\widetilde{F}=s^w u=t^w x^w u$. Also, in order for $\widetilde{F}$ to be smooth up until $\rho=0$, we require that $u$ be smooth until the boundary. Thus, the two extension problems are equivalent, and the normalized obstruction operators must agree.
\end{proof}

\begin{theorem}
Let $(X^{d+1},g_+,f_+,m,\mu)$ be a weighted Poincar\'e space for the smooth metric measure space $(M^d,g,f,m,\mu)$. Let $k \in \mathbb{N}$, $k \leq (d+m)/2$ for $d+m\in 2\mathbb{N}$, and set $s=(d+m)/2+k$. Let $v_{1}, v_{2} \in C^{\infty}(M)$ and let $u_{1}, u_{2}$ denote the corresponding solutions of 
$$
\left((\Delta_{\phi_+}){}_{g_+}-s(d+m-s)\right) u=O\left({r^{2k}}\log r\right)
$$
given by \cref{differential-equation-solution}. Then for fixed small {$r_{0}>0$}
\begin{multline}
\label{21}
\operatorname{lp} \int_{\epsilon<{r}<{r_{0}}}\left[\left\langle d u_{1}, d u_{2}\right\rangle_{g_+}-s(d+m-s) u_{1} {u_{2}}\right] \;(d v_\phi^m)_{g_+}\\
=-d_k\int_{M}\left[\left(\frac{d+m}{2}+k\right) v_{1} L_{2k,\phi}^m v_2+\left(\frac{d+m}{2}-k\right)  {v_{2}}L_{2k,\phi}^m v_1\right] (d v_\phi^m)_{g},
\end{multline}
where $\mathrm{lp}$ denotes the coefficient of $\log \epsilon$ in the asymptotic expansion of the integral as $\epsilon \rightarrow 0$, $d_k=[2^{2k-1}k!(k-1)!]^{-1}$, and $d v_\phi^m=e^{-\phi}\mathrm{dvol}$ denotes the weighted volume element. In particular, $L_{2k,\phi}^m$ is formally self-adjoint.
\end{theorem}

\begin{proof}
For $(g,f)$, we may assumed without loss of generality~{\cites{CaseKhaitan2022} (Proposition 5.3)} that we have a normal weighted Poincar\'e metric, i.e. $(g_+,f_+)=(r^{-2}(dr^2+g_r),r^{-1}f_r)$. Green's identity gives
\begin{align*}
&\int_{\epsilon<r<r_{0}}\left[\left\langle d u_{1}, d u_{2}\right\rangle_{g_+}-s(d+m-s) u_{1} {u_{2}}\right] \;(d v_\phi^m)_{g_{+}} \\
&=-\epsilon^{1-d-m} \oint_{r=\epsilon} u_{1} \partial_{r} {u_{2}} \;(d v_\phi^m)_{g_{r}}+O(1).
\end{align*}

Substituting
\begin{align*}
u_i&=r^{\frac{d+m}{2}-k}\left(V_i+d_k L_{2k,\phi}^m v_i \,r^{2 k} \log r\right),\\
(d v_\phi^m)_{g_{r}}&=\left(1+(v_\phi^m)_2 r^{2}+(\text {even powers})+\ldots\right) (d v_\phi^m)_{g},
\end{align*}
and expanding shows that the coefficient of $\log \epsilon$ in the expansion of this expression is
$$
-d_k\int_{M}\left[\left(\frac{d+m}{2}+k\right) v_{1} { L_{2k,\phi}^m v_2}+\left(\frac{d+m}{2}-k\right) {v_{2}} L_{2k,\phi}^m v_1 \right] (d v_\phi^m)_{g}.
$$
The symmetry of the left-hand side of \cref{21} yields the final conclusion.
\end{proof}
\bibliographystyle{plain}
\bibliography{bib.bib}

\end{document}